\documentclass[a4paper,11pt,openbib,reqno]{amsart}
\usepackage{amsmath,amsfonts,amssymb}
\usepackage{verbatim}
\usepackage{enumerate}
\usepackage{url}
\usepackage[colorlinks]{hyperref}
\usepackage[latin1]{inputenc}
\usepackage{enumitem}
\usepackage[english]{babel}
\usepackage{tikz}
\usepackage[margin=3.5cm]{geometry}
\usepackage{bbm}
\usepackage{multirow}
\usepackage{mathrsfs}
\usepackage{caption}
\usepackage{float}
\restylefloat{table}
\usepackage{graphicx}

\def\Z{\mathbb Z}

\def\Ca{\mathcal C}

\def\ha{\mathcal H}

\def \F{\mathbb F}

\def\Fq{\mathbb{F}_q}
\def\Fp{\mathbb{F}_p}
\def\Fqn{\mathbb{F}_{q^n}}

\def\Fqdo{\mathbb{F}_{q^2}}

\DeclareMathOperator{\tr}{Tr}

\def\lcm{\mathop{\rm lcm}}

\theoremstyle{plain}
\newtheorem{theorem}{Theorem}[section]
\newtheorem{lemma}[theorem]{Lemma}
\newtheorem{definition}[theorem]{Definition}
\newtheorem{corollary}[theorem]{Corollary}
\newtheorem{proposition}[theorem]{Proposition}

\newtheorem{remark}[theorem]{Remark}
\newtheorem{example}[theorem]{Example}

\newtheorem{problem}[theorem]{Problem}

\theoremstyle{definition}

\author[J. A. Oliveira]{Jos\'e Alves Oliveira}
\address{Departamento de Matem\'{a}tica,
Universidade Federal de Minas Gerais,
UFMG,
Belo Horizonte MG (Brazil),
 30123-970}

\email{joseufmg@gmail.com}

\date{\today
}

\title{On diagonal equations over finite fields}

\begin{document}

\begin{abstract}
Let $\Fq$ be a finite field with $q=p^t$ elements. In this paper, we study the number of solutions of equations of the form $a_1 x_1^{d_1}+\dots+a_s x_s^{d_s}=b$ over $\Fq$. A classic well-konwn result from Weil yields a bound for such number of solutions. In our main result we give an explicit formula for the number of solutions of diagonal equations satisfying certain natural restrictions on the exponents. In the case $d_1=\dots=d_s$, we present necessary and sufficient conditions for the number of solutions of a diagonal equation being maximal and minimal with respect to Weil's bound. In particular, we completely characterize maximal and minimal Fermat type curves. 
\end{abstract}
\maketitle

\section{Introduction}

Let $p$ be a prime number and $t$ be a positive integer. Let $\Fq$ be a finite field with $q$ elements, where $q=p^t$.  For $\vec{a}=(a_1,\dots,a_s)\in\Fq^s$, $\vec{d}=(d_1,\dots,d_s)\in\Z_+^s$ and $b\in\Fq$, let $N_s(\vec{a},\vec{d},q,b)$ be the number of solutions of the diagonal equation
\begin{equation}\label{item200}
a_1 x_1^{d_1}+\cdots+a_s x_s^{d_s}=b
\end{equation}
over $\Fq$. In the general case,  Weil~\cite{weil1949numbers} and Hua and Vandiver~\cite{hua1949characters} independently showed that $N_s(\vec{a},\vec{d},q,b)$ can be expressed in terms of character sums. For the case $b=0$, Weil's result implies that
\begin{equation}\label{item223}
|N_s(\vec{a},\vec{d},q,0)-q^{s-1}|\le I(d_1,\dots,d_s)(q-1)q^{(s-2)/2}
\end{equation}
where $I(d_1,\dots,d_s)$ is the number of $s$-tuples $(y_1,\dots,y_s)\in\Z^n$, with $1\leq y_i\leq d_i-1$ for all $i=1,\dots,s$, such that
\begin{equation}\label{item220}
\frac{y_1}{d_1}+\dots+\frac{y_s}{d_s}\equiv 0\pmod{1}.
\end{equation}
 A formula for $I(d_1,\dots,d_s)$ can be found in Lidl and
Niederreiter~\cite[p. 293]{Lidl}. Some properties of $I(d_1,\dots,d_s)$ have been explored for several authors~\cite{cao2007factorization,sun1987solvability} and the possible values of $N_s(\vec{a},\vec{d},q,0)$ in the case where $I(d_1,\dots,d_s)\in\{1,2\}$ was studied by Sun and Yuan \cite{sun1996number}. A diagonal equation given by Eq.~\eqref{item200} (with $b=0$) is called {\it maximal} (or {\it minimal}) if its number of solutions attains the bound \eqref{item223} and the maximality or minimality are set accordingly to $N_s(\vec{a},\vec{d},q,0)$ attaining the upper or lower bound, respectively. The number of solutions of diagonal equations, with $s=3$, $d_1=d_2=d_3$ and $b=0$, is closely related to the number of $\Fq$-rational points on curves of the form $ax^n+by^n=c$ (see Section~\ref{item221} for more details). Maximality and minimality have been extensively studied in the context of curves \cite{tafazolian2014curve,garcia2008certain,tafazolian2013maximal}. For instance, maximal and minimal Fermat type curves of the form $x^n+y^n=1$ were studied by Tafazolian \cite{tafazolian2010characterization}.

 The number of solutions of diagonal equations have been extensively studied in the last few decades \cite{cao2007factorization,baoulina2010number,hou2009certain,sun1996number,qi1997diagonal,baoulina2016class,zhou2019counting}. In many cases, the authors present a formula for the number of solutions of equations whose exponents satisfy certain natural restrictions. The case where $q$ is a square provides families of diagonal equations whose number of points can be obtained by means of simple parameters. For instance, in the case where $q=p^{2t}$, Wolfmann~\cite{wolfmann1992number} presented an explicit formula for $N_s(\vec{a},\vec{d},q,b)$ in the case where $d=d_1=\dots=d_s$ and there exists a divisor $r$ of $t$ such that $d$ divides $p^r+1$. Still in the case where $q=p^{2t}$, Cao, Chou and Gu~\cite{cao2016number} obtained a formula for $N_s(\vec{a},\vec{d},q,b)$ in terms of $I(d_1,\dots,d_s)$ in the case where there exists a divisor $r$ of $t$ such that $d_i$ divides $p^r+1$ and $a_i\in\F_{p^r}$ for all $i=1,\dots,s$. For more results concerning diagonal equations over finite fields, see Section 7.3 in \cite{Panario} and the references therein.

In this paper, we obtain an explicit formula for $N_s(\vec{a},\vec{d},q,b)$ in a setting more general than that presented in \cite{wolfmann1992number} and \cite{cao2016number}. In Theorem \ref{item227}, we present the number of solutions of Eq.~\eqref{item200} in the case where $b= 0$, $q=p^{2t}$ and, for each $i=1,\dots,s$, there exists a divisor $r_i$ of $t$ such that $d_i|(p^{r_i}+1)$. Most notably, Theorem \ref{item201} provides the number of solutions of Eq.~\eqref{item200} in the case where $b\neq 0$, $q=p^{2t}$ and  there exists a divisor $r$ of $t$ such that $d_i|(p^r+1)$ for all $i=1,\dots,s$. As a consequence of our results, a simple formula for $I(d_1,\dots,d_s)$ is obtained. In the case $d_1=\dots=d_s$, we study the number of solutions of \eqref{item200} in order to find those diagonal equations whose number of solutions attains the bound \eqref{item220}. In Theorem~\ref{item219} we provide necessary and sufficient conditions on $a_1,\dots,a_s$ and $q$ for the diagonal equation of the form \eqref{item200} being maximal or minimal. In particular, we prove that a maximal (or minimal) diagonal equation must satisfy the hypothesis of Theorem~\ref{item201}. As a direct consequence of Theorem~\ref{item219} we obtain a complete characterization of maximal and minimal curves of the form $ax^n+by^n=c$. In particular, we prove that the curve with affine equation $ax^n+by^n=c$ is maximal only if it is covered by the Hermitian curve. We also discuss further problems concerning the number of solutions of diagonal equations in the projective space.


The paper is organized as follows. In Section~\ref{item243} we state our main results and provide some important remarks. Section~\ref{item244} provides preliminary results concerning Jacobi sums. In Section~\ref{item245} we prove our main counting results. In Section~\ref{item242} we study the conditions in which a diagonal equation attains Weil's bound and in Section~\ref{item221} we give a characterization of maximal and minimal projective varieties arising from diagonal equations. Finally, in Section~\ref{item246} we provide some final considerations and open problems. 

\section{Main results}\label{item243}
 In this section we state the main results of this article. Throughout the paper, unless otherwise stated,  $q=p^t$ for some positive integer $t$ and an odd prime $p$ and $\alpha$ is a primitive element of $\Fqdo$. Let $n$ be a positive integer. If $d_i=1$ for some $i\in\{1,\dots,s\}$, then the number of solutions of Eq.~\eqref{item200} over $\Fqn$ is $q^{n(s-1)}$, then along the paper we assume that $d_i>1$ for all $i=1,\dots,s$. Moreover, by simple change of variables, we may assume without loss of generality that $d_i$ is a divisor of $q^n-1$ for all $i=1,\dots,s$. Also, we let $\vec{a}$ denote a vector $(a_1,\dots,a_s)\in\Fqn^s$, where $a_i\neq 0$ for all $i=1,\dots,s$. The following definitions will be extensively used in our main results.

\begin{definition}
	For $\vec{a}=(a_1,\dots,a_s)\in\Fqn^s$ and $\vec{d}=(d_1,\dots,d_s)$, let $N_s(\vec{a},\vec{d},q^n,b)$ be the number of solutions of the diagonal equation 
	\begin{equation}\label{item257}
	a_1 x_1^{d_1}+\cdots+a_s x_s^{d_s}=b
	\end{equation}
	over $\F_{q^n}$.
\end{definition}
Throughout the paper, unless otherwise stated, $N_s(\vec{a},\vec{d},q^n,b)$ denotes the number
of solutions of Eq.~\eqref{item257} over $\Fqn$ with $\vec{a}\in\Fqn^s$ and $d_i|(q^n-1)$ for all $i=1,\dots,s$.

\begin{definition}For $d$ a divisor of $q^2-1$ and $a,b\in\Fqdo^*$, we set
	$$\theta_d(a,b)=\begin{cases}
	1,&\text{ if }a^{(q^2-1)/d}=b^{(q^2-1)/d};\\
	0,&\text{ otherwise.}
	\end{cases}$$
\end{definition}

Our main results can be summarized as follows.

\begin{theorem}\label{item227}
	 Let $\vec{a}\in\Fqdo^s$. Assume that for each $i$, with $i=1,\dots,s$, there exists a divisor $r_i$ of $t$ such that $d_i|(p^{r_i}+1)$. For each $i=1,\dots,s$, let $\lambda_i=\alpha^{(p^t+1)/2}$ if $d_i|(p^t+1)$ and let $\lambda_i=1$ otherwise. Then 
	$$N_s(\vec{a},\vec{d},q^2,0)=q^{2(s-1)}+ q^{s-2}\sum_{j=1}^{q^2-1}\prod_{i=1}^{s}\varepsilon_i(1-d_i)^{\Delta_{i,j}},$$
	 where $\varepsilon_i=(-1)^{t/{r_i}}$ and $\Delta_{i,j}=\theta_{d_i}(a_i,\lambda_i\alpha^j)$.
\end{theorem}
The following results are generalizations of Theorem 1 in Wolfmann~\cite{wolfmann1992number} and Theorem 2.9 in Cao, Chou and Gu~\cite{cao2016number}.
\begin{corollary}\label{item250}
	Let $\vec{a}\in\Fqdo^s$. If there exists a divisor $r$ of $t$ such that $d_i|(p^r+1)$ for all $i=1,\dots,s$, then 
	$$N_s(\vec{a},\vec{d},q^2,0)=q^{2(s-1)}+\varepsilon^s q^{s-2}(q+\varepsilon)\sum_{j=1}^{q-\varepsilon}\prod_{i=1}^{s}(1-d_i)^{\delta_{i,j}},$$
	where $\delta_{i,j}=\theta_{d_i}(a_i,\alpha^j)$ and $\varepsilon=(-1)^{t/r}$.
\end{corollary}
For the case $b\neq0$, we have the following result.
\begin{theorem}\label{item201} Let $\vec{a}\in\Fqdo^s$. If there exists a divisor $r$ of $t$ such that $d_i|(p^r+1)$ for all $i=1,\dots,s$, then
	\begin{equation}\label{item218}
	N_s(\vec{a},\vec{d},q^2,b)=q^{2(s-1)}-\varepsilon^{s+1}q^{s-2}\left(q\prod_{i=1}^{s}(1-d_i)^{\nu_i(b)}-\sum_{j=1}^{q-\varepsilon}\prod_{i=1}^{s}(1-d_i)^{\delta_{i,j}}\right)
	\end{equation}
	for $b\neq 0$, where $\delta_{i,j}=\theta_{d_i}(a_i,\alpha^j)$, $\nu_{i}(b)=\theta_{d_i}(a_i,b)$ and $\varepsilon=(-1)^{t/r}$.
\end{theorem}

In the special case where $s=2$, we have the following result for the number of points on Fermat type curves.

\begin{corollary}\label{item202}
	Let $a,b\in\Fqdo^*$ and $c\in\Fqdo$. Let $m$ and $n$ be divisors of $q^2-1$ and let $N(c)$ be the number of $\Fqdo$-rational points on the Fermat type curve given by the affine equation $a x^n+b y^m=c$. If there exists a divisor $r$ of $t$ such that $n|(p^r+1)$ and $m|(p^r+1)$, then
	$$N(0)=q^2-(q^2-1)(1-l)^{\theta_l(a,b)}+1-C(n,m)$$
	and
	$$N(c)=q^2-\varepsilon q(1-n)^{\theta_n(a,c)}(1-m)^{\theta_m(b,c)}-\varepsilon(q-\varepsilon)(1-l)^{\theta_l(a,b)}+1-C(n,m)$$
	for $c\neq 0$, where $l=\gcd(m,n)$, $\varepsilon=(-1)^{t/r}$ and $C(n,m):=(1-n)^{\theta_n(a,b)}$ if $n=m$ and $C(n,m):=0$ otherwise.
\end{corollary}

\begin{remark}\label{item249} Let $\vec{a}\in\Fq$ and let $\vec{d}=(d_1,\dots,d_s)$ such that $d_i|(q-1)$ for all $i=1,\dots,s$. The bounds provided by Weil~\cite{weil1949numbers} entail that
	\begin{enumerate}[label=(\alph*)]
		\item\label{item236} $\left|N_s(\vec{a},\vec{d},q,0)-q^{s-1}\right|\le I(d_1,\dots,d_s)(q-1) q^{(s-2)/2} $;\vskip0.1cm
		\item\label{item237}  $\left|N_s(\vec{a},\vec{d},q,b)-q^{s-1}\right|\le  q^{(s-2)/2}\left[\sqrt{q}\prod_{i=1}^{s}(d_i-1)-(\sqrt{q}-1)I(d_1,\dots,d_s)\right]$.
	\end{enumerate}
\end{remark}

Diagonal equations are called {\it maximal} or {\it minimal} if its number of solutions attains the upper or the lower bound, respectively. 

\begin{lemma}\label{item251}
	 For $\vec{d}\in\Z^s$ such that $d_i|(q-1)$ for all $i=1,\dots,s$, it follows that
	$$I(d_1,\dots,d_s)=\tfrac{(-1)^s}{q-1}\sum_{m=0}^{q-1}\prod_{d_i|m} (1-d_i).$$
\end{lemma}

Theorem \ref{item227} provides the number os solutions of diagonal equations under conditions on $d_1,\dots,d_s$ and $q$. Among the family of diagonal curves covered by this result there exist those whose number of solutions attains Weil's bound (see Remark~\ref{item233}). Indeed, if $d_1=\dots=d_s$, then a diagonal equation attaining Weil's bound satisfies the conditions imposed in Theorem \ref{item201}, as we state in the following result.

\begin{theorem}\label{item219}
	For $d$ a divisor of $q-1$, let $\chi_d$ be a multiplicative character of $\Fq^*$ of order $d$. Assume that $d_1=\dots=d_s=d>2$, $(s,b)\neq(2,0)$, $(s,d,b)\neq(4,3,0)$ and $(s,d)\neq(3,3)$ if $b\neq 0$. Then the bounds in Remark~\ref{item249} are attained if and only if the following hold:
	\begin{itemize}
		\item $q=p^{2t}$ for some positive integer $t$;
		\item there exists a divisor $r$ of $t$ such that $d|(p^{r}+1)$;
		\item if $b=0$, then $\chi_{d}(a_1)=\dots=\chi_{d}(a_s)$;
		\item if $b\neq 0$, then $\chi_{d}(a_1)=\dots=\chi_{d}(a_s)=\chi_d(b)$.
	\end{itemize}
	Suppose Weil's bound is attained. If $b=0$, then the diagonal equation is minimal if and only if $t/r$ is even and $s$ is odd. If $b\neq 0$, then the diagonal equation is minimal if and only if $t/r$ and $s$ are both even. 
\end{theorem}

The case $(s,b)=(2,0)$ is commented in Remark~\ref{item248}. The cases $(s,d,b)=(4,3,0)$ and $(s,d)\neq(3,3)$ with $b\neq 0$ are not included in Theorem~\ref{item219} because of a technical obstruction in Lemma~\ref{item234}. The number of solutions of diagonal equations of degree $d=2$ is well-known (see Theorems 6.26 and 6.27 in~\cite{Lidl}). As a direct consequence of Theorem~\ref{item219}, we have the following characterization for maximal and minimal Fermat type curves.

\begin{corollary}\label{item222}
	For $n$ a divisor of $q-1$, let $\chi_n$ be a multiplicative character of $\Fq^*$ of order $n$. For $a,b,c\in\Fqdo$ and $n$ a divisor of $q^2-1$, let $\mathcal{C}$ be the curve $ax^n+by^n=c$ over $\F_{q^2}$ with $q=p^t$. Then
	\begin{enumerate}
		\item $\mathcal{C}$ is maximal over $\F_{q^2}$ if and only if  the following hold:
		\begin{itemize}
			\item $n$ divides $q+1$;
			\item $\chi_n(a)=\chi_n(b)=\chi_n(c)$.
		\end{itemize}
		\item  $\mathcal{C}$ is minimal over $\F_{q^2}$ if and only if the following hold:
		\begin{itemize}
			\item $t$ is even and there exists a divisor $r$ of $t/2$ such that $n$ divides $p^r+1$;
			\item $\chi_n(a)=\chi_n(b)=\chi_n(c)$.
		\end{itemize}
	\end{enumerate}
\end{corollary}

Corollary~\ref{item222} generalizes the main result of Tafazolian~\cite{tafazolian2010characterization} and also generalizes Theorem 4.4 of Garcia and Tafazolian~\cite{garcia2008certain}, where the authors study maximal curves of the form $x^n+y^n=1$. In particular, Corollary~\ref{item222} implies that $\mathcal{C}$ is maximal (or minimal) only if it is covered by a Hermitian curve. In Section~\ref{item221} we provide a more general approach on the number of points on diagonal equations in the $s$-dimensional projective space. In particular, Fermat type varieties attaining the Weil-Deligne bound are completely characterized (see Corollary~\ref{item241}).

\section{Preliminaries}\label{item244}
Let $\Ca$ be a non-singular and geometrically irreducible curve defined over $\Fq$ and let $\Ca(\Fqn)$ denote the set of $\Fqn$-rational points on $\Ca$. The Hasse-Weil bound asserts that 
$$|\Ca(\Fqn)-q^n-1|\le 2g\sqrt{q^n},$$
where $g$ is the genus of $\Ca$. The curve $\Ca$ is called \textit{maximal} if 
$$\Ca(\Fqn)=q^n+1+2g\sqrt{q^n}.$$
For $r$ a positive integer, let $\ha_r$ be the Hermitian curve over $\F_{p^{2r}}$ given by the affine equation $x^{p^r+1}+y^{p^r+1}=1$. It is well-known that $\ha_r$ is maximal over $\F_{p^{2r}}$ and, since $g=p^r(p^r-1)/2$, the number of rational points on $\ha_r$ over $\F_{p^{2r}}$ is $p^{2r}+1+p^r(p^r-1)p^r=p^{3r}+1$. For $\lambda_1,\ldots, \lambda_k$ multiplicative characters of $\Fq^*$, the \textit{Jacobi sum} of $\lambda_1,\ldots, \lambda_k$ is defined as 
$$J_{q}(\lambda_1,\ldots,\lambda_k, b)=\sum_{b_1+\cdots b_k=b}\lambda_1(b_2)\cdots\lambda_k(b_k)$$
where the summation is extended over $k$-tuples $(b_1,\ldots,b_k)\in\Fq^k$. The following useful result is a basic result of Jacobi sums.

\begin{proposition}\cite[Theorems $5.20$ and $5.22$]{Lidl}\label{item203} Let $\lambda_1,\ldots,\lambda_k$ be nontrivial multiplicative characters of $\Fq^*$. Then 
	$$|J_{q}(\lambda_1,\ldots,\lambda_k, b)|=\begin{cases}
	q^{\frac{k-1}{2}},&\text{ if }b\neq0\text{ and } \lambda_1\cdots\lambda_k\text{ is nontrivial;}\\
	q^{\frac{k-2}{2}},&\text{ if }b\neq0\text{ and } \lambda_1\cdots\lambda_k\text{ is trivial;}\\
	(q-1)q^{\frac{k-2}{2}},&\text{ if }b=0\text{ and } \lambda_1\cdots\lambda_k\text{ is trivial;}\\
	0,&\text{ if }b=0\text{ and } \lambda_1\cdots\lambda_k\text{ is nontrivial.}\\
	\end{cases}$$
	
\end{proposition}

From here, we use the maximality of the Hermitian curve in order to compute certain Jacobi sums.

\begin{lemma}\label{item206}
	Let $q=p^{t}$ and let $r$ be a divisor of $t$. Let $m$ and $n$ be divisors of $q^2-1$ such that $m|(p^r+1)$ and $n|(p^r+1)$ and let $\chi_m$ and $\chi_n$ be primitive multiplicative characters of $\Fqdo^*$ of order $m$ and $n$, respectively. For $\ell_1,\ell_2$ integers, it follows that
	$$J_{q^2}\big(\chi_{n}^{\ell_1},\chi_{m}^{\ell_2},1\big)=\begin{cases}
	\varepsilon q,\text{ if } \chi_n^{\ell_1}\chi_m^{\ell_2}\text{ is nontrivial;}\\
	-1,\text{ otherwise,}\\
	\end{cases}$$
	where $\varepsilon=(-1)^{t/r}$.
\end{lemma}

\begin{proof}
	If $\chi_n^{\ell_1}\chi_m^{\ell_2}$ is trivial, then
	$$J_{q^2}(\chi_{n}^{\ell_1},\chi_{m}^{\ell_2},1)=\sum_{b_1+b_2=1} \chi_{n}^{\ell_1}(b_1)\chi_{n}^{-\ell_1}(b_2)=\sum_{\substack{b_1+b_2=1\\b_1\neq 1}} \chi_{n}^{\ell_1}\left(\tfrac{b_1}{1-b_1}\right)=-\chi_{n}^{\ell_1}(-1)=-1.$$
	
	Suppose that $\chi_n^{\ell_1}\chi_m^{\ell_2}$ is nontrivial. We observe that $\chi_n^{\ell_1}=\chi_{p^r+1}^{(p^r+1)\ell_1/n}$ and $\chi_m^{\ell_2}=\chi_{p^r+1}^{(p^r+1)\ell_2/m}$ and then we may suppose without loss of generality that $n=m=p^r+1$. Since $\ha_r$ is maximal over $\F_{p^{2r}}$, it follows that $\ha_r(\Fqdo)=p^{2t}+1-\varepsilon p^r(p^r-1)p^t$, where $\varepsilon$ is $1$ or $-1$ according to the minimality or maximality of $\ha_r$ over $\Fq$. The line at infinity contains $p^r+1$ points, then the number $N$ of solutions of the affine equation $x^{p^r+1}+y^{p^r+1}=1$ over $\Fqdo$ is equal to
	\begin{equation}\label{item204}
	p^{2t}+1-\varepsilon p^r(p^r-1)p^t-(p^r+1)=p^{2t}-p^r-\varepsilon p^r(p^r-1)p^t.
	\end{equation}
	On the other hand, we have that
	$$\begin{aligned}
	N&=\sum_{b_1+b_2=1}\left[1+\chi_{p^r+1}(b_1)+\cdots+\chi_{p^r+1}^{p^r}(b_1)\right] \left[1+\chi_{p^r+1}(b_2)+\cdots+\chi_{p^r+1}^{p^r}(b_2)\right]\\
	&=p^{2t}+\sum_{b_1+b_2=1}\sum_{1\le \ell_1,\ell_2\le p^r} \chi_{p^r+1}^{\ell_1}(b_1)\chi_{p^r+1}^{\ell_2}(b_2)\\
	&=p^{2t}+\sum_{\substack{1\le \ell_1,\ell_2\le p^r \\ \ell_2\neq p^r+1-\ell_1}} \sum_{b_1+b_2=1}\chi_{p^r+1}^{\ell_1}(b_1)\chi_{p^r+1}^{\ell_2}(b_2)+\sum_{1\le\ell_1\le p^r}\sum_{b_1\in\Fqdo\backslash\{1\}} \chi_{p^r+1}^{\ell_1}\left(\tfrac{b_1}{1-b_1}\right)\\
	&=p^{2t}+\sum_{\substack{1\le \ell_1,\ell_2\le p^r \\ \ell_2\neq p^r+1-\ell_1}}J_{q^2}(\chi_{p^r+1}^{\ell_1},\chi_{p^r+1}^{\ell_2},1)-\sum_{1\le\ell_1\le p^r}\chi_{p^r+1}^{\ell_1}(-1)\\
	&=p^{2t}-p^r+\sum_{\substack{1\le \ell_1,\ell_2\le p^r \\ \ell_2\neq p^r+1-\ell_1}} J_{q^2}(\chi_{p^r+1}^{\ell_1},\chi_{p^r+1}^{\ell_2},1).\\
	\end{aligned}$$
	Since the summation contains $p^r(p^r-1)$ elements, it follows from Proposition~\ref{item203} and Eq.~\eqref{item204} that $J_{q^2}(\chi_{p^r+1}^{\ell_1},\chi_{p^r+1}^{\ell_2},1)=-\varepsilon p^t$ and this proves the assertion of our result.
\end{proof}

Another proof for Lemma \ref{item206} can be obtained using Theorems 5.16, 5.21 and 5.26 of Lidl and Niederreiter~\cite{Lidl}.

\section{On the number of solutions}\label{item245}

In this section, we prove our counting results. For proving Theorem \ref{item227}, we will follow the main ideas of Wolfmann~\cite{wolfmann1992number}. To this end, let $a\in\Fqdo$ and let $\psi_a=\exp\left((2\pi i)\tr(ax)/p\right)$ be an additive character, where $\tr$ denotes the trace from $\F_{q^2}$ to $\Fp$. We have the following known results.

\begin{lemma}\label{item228}
	Let $\vec{a}=(a_1,\dots,a_s)\in\Fqdo^s$ and $\vec{d}=(d_1,\dots,d_s)\in\Z_+^s$ such that $d_i|(q^2-1)$ for all $i=1,\dots,s$. Then 
	$$N_s(\vec{a},\vec{d},q^2,b)=q^{-2}\sum_{c\in\scalebox{0.7}{$\displaystyle \Fqdo^*$}} \psi_c(-b)\prod_{i=1}^{s}S_{i}(c),$$
	where $S_{i}(c):=\sum_{x\in\scalebox{0.7}{$\displaystyle \Fqdo^*$}}\psi_{ca_i}\left(x^{d_i}\right)$.
\end{lemma}
\begin{proof}
	It follows by a similar argument to the proof of Proposition 1 of \cite{wolfmann1992number}.
\end{proof}

\begin{lemma}\cite[Corollary 3]{wolfmann1992number}\label{item229}
	Let $q=p^{t}$ and $\vec{a},\vec{d}$ and $S_{i}(c)$ as defined in Lemma~\ref{item228}. Suppose that there exists a divisor $r_i$ of $t$ such that $d_i|(p^{r_i}+1)$ for each $i=1,\dots,s$. Let $\varepsilon_i=(-1)^{t/r_i}$ and $\epsilon_i=\varepsilon_i^{(p^{r_i}+1)/d_i}$. Then, for each $i$, we have that
	\begin{enumerate}[label=(\alph*)]
		\item If $(ca_i)^{(q^2-1)/d_i}=\epsilon_i$, then $S_i(c)=-\varepsilon_i(d_i-1)q$;\\
		\item If $(ca_i)^{(q^2-1)/d_i}\neq \epsilon_i$, then $S_i(c)= \varepsilon_i\, q $.\\
	\end{enumerate}
\end{lemma}
From Lemmas \ref{item228} and \ref{item229}, we are able to prove Theorem \ref{item227}.

\subsection{Proof of Theorem \ref{item227}} Since $b=0$, it follows from Lemma~\ref{item228} that 
	$$N_s(\vec{a},\vec{d},q^2,0)=q^{-2}\sum_{c\in\scalebox{0.7}{$\displaystyle \Fqdo^*$}} \prod_{i=1}^{s}S_{i}(c).$$
	We observe that $S_i(0)=\sum_{x\in\Fqdo}\psi_{0\cdot a_i}\left(x^{d_i}\right)=\sum_{x\in\Fqdo}\exp(0)=q^2.$ Therefore
	\begin{equation}\label{item230}
	N_s(\vec{a},\vec{d},q^2,0)=q^{-2}\prod_{i=1}^{s}q^2+q^{-2}\sum_{c\in\scalebox{0.7}{$\displaystyle \Fqdo^*$}} \prod_{i=1}^{s}S_{i}(c)=q^{2(s-1)}+q^{-2}\sum_{c\in\in\scalebox{0.7}{$\displaystyle \Fqdo^*$}} \prod_{i=1}^{s}S_{i}(c^{-1}).
	\end{equation}
	By Lemma~\ref{item229}, it follows that 
	$$S_i(c^{-1})=\begin{cases}
	-\varepsilon_i(d_i-1)q,&\text{ if }a_i^{(q^2-1)/d_i}=\epsilon_ic^{(q^2-1)/d_i};\\
	\varepsilon_i\, q,& \text{ if }a_i^{(q^2-1)/d_i}\neq\epsilon_ic^{(q^2-1)/d_i},\\
	\end{cases}$$
	where $\varepsilon_i:=(-1)^{t/r_i}$ and $\epsilon_i:=\varepsilon_i^{(p^{r_i}+1)/d_i}$. Since there exists $r_i$ a divisor of $t$ such that $d_i|(p^{r_i}+1)$, it follows that $\epsilon_i=-1$ if and only if $d_i|(p^t+1)$ and $(p^{r_i}+1)/d_i$ is odd. Furthermore, if $d_i|(p^t+1)$, then $\left(\alpha^{(p^t+1)/2}\right)^{(q^2-1)/d_i}=\epsilon_i$. Altogether, we have shown that 
	\begin{equation}\label{item231}
	S_i(c^{-1})=\varepsilon_i(1-d_i)^{\theta_{d_i}(a_i,\lambda_ic)}q,
	\end{equation}
	where $\lambda_i=\alpha^{(p^t+1)/2}$ if $d_i|(p^t+1)$ and $\lambda_i=1$ otherwise. Eq.~\eqref{item230} and \ref{item231} imply that
	$$N_s(\vec{a},\vec{d},q^2,0)=q^{2(s-1)}+q^{s-2}\sum_{j=1}^{q^2-1} \prod_{i=1}^{s}\varepsilon_i(1-d_i)^{\theta_{d_i}(a_i,\lambda_i \alpha^j)},$$
	proving our result.$\hfill\qed$
	
	\begin{remark}\label{item233}
		Assume that $d>2$. From Theorem~\ref{item227}, it can be verified that 
		$$\left|N_s(\vec{a},\vec{d},q^2,0)-q^{2(s-1)}\right|\le q^{s-2}\left|\sum_{m=1}^{q^2-1}\prod_{d_i|m} (1-d_i)\right|,$$
		for $\vec{a}$ and $\vec{d}$ satisfying the hypothesis of Theorem~\ref{item227}.
	\end{remark}
	
	\begin{example}
	Let $d=(d_1,\dots,d_s)\in\Fqdo^s$ be a $s$-tuple satisfying the hypothesis of Theorem~\ref{item227} and such that $d_1,\dots,d_k$ are divisors of $q-1$ and $d_{k+1},\dots,d_s$ are divisors of $q+1$. Let $\lambda=\alpha^{(q+1)/2}$. By Theorem~\ref{item227}, the number of solutions of the diagonal equation
	$$x^{d_1}+\dots+x^{d_k}+\lambda x^{d_{k+1}}+\dots+\lambda x^{d_{s}}=0$$
	attains the bound in Remark~\ref{item233}.
	\end{example}

\subsection{Proof of Corollary~\ref{item250}}
	By Theorem~\ref{item227}, it follows that
	$$N_s(\vec{a},\vec{d},q^2,0)=q^{2(s-1)}+ q^{s-2}\varepsilon^s\sum_{j=1}^{q^2-1}\prod_{i=1}^{s}(1-d_i)^{\Delta_{i,j}},$$
	where $\varepsilon=(-1)^{t/{r}}$, $\Delta_{i,j}=\theta_{d_i}(a_i,\lambda_i\alpha^j)$ and $\lambda_i=\alpha^{(p^t+1)/2}$ if $d_i|(p^t+1)$ and $\lambda_i=1$ otherwise. We observe that $(p^r+1)|(p^t+1)$ if and only if $t/r$ is odd. If $t/r$ is odd, then $\lambda_i=1$ for all $i=1,\dots,s$. If $t/r$ is even, then $\lambda_i=\alpha^{(p^t+1)/2}$ for all $i=1,\dots,s$ and so $\Delta_{i,j}=\theta_{d_i}(a_i,\alpha^{(p^t+1)/2+j})=\Delta_{i,(p^t+1)/2+j}$. In both cases, we have that
	\begin{equation}\label{item258}
	N_s(\vec{a},\vec{d},q^2,0)=q^{2(s-1)}+ q^{s-2}\varepsilon^s\sum_{j=1}^{q^2-1}\prod_{i=1}^{s}(1-d_i)^{\Delta_{i,j}}=q^{2(s-1)}+ q^{s-2}\varepsilon^s\sum_{j=1}^{q^2-1}\prod_{i=1}^{s}(1-d_i)^{\delta_{i,j}},
	\end{equation}
	where $\delta_{i,j}:=\theta_{d_i}(a_i,\alpha^j)$. We observe that $d_i|(p^r+1)$ implies that $d_i|(q-\varepsilon)$. Therefore $\delta_{i,j}=\delta_{i,k}$ if $j\equiv k\pmod{q-\varepsilon}$ and then the assertion of our result follows by Eq.~\eqref{item258} .$\hfill\qed$\\

In this section, for $l$ a divisor of $q^2-1$, let $\chi_l$ be a multiplicative character of $\Fqdo^*$ of order $l$. The first step in order to prove Theorem~\ref{item201} is the following result.

\begin{proposition}\label{item212}
	If there exists a divisor $r$ of $t$ such that $d_1|(p^r+1)$ and $d_2|(p^r+1)$, then
	$$N_2(\vec{a},\vec{d},q^2,0)=q^2-(q^2-1)(1-l)^{\theta_l(a_1,a_2)}$$
	and
	$$N_2(\vec{a},\vec{d},q^2,c)=q^2-\varepsilon q(1-d_1)^{\theta_{d_1}(a_1,c)}(1-d_2)^{\theta_{d_2}(a_2,c)}-\varepsilon(q-\varepsilon)(1-l)^{\theta_l(a_1,a_2)}$$
	for $c\neq0$, where $l:=\gcd(d_1,d_2)$ and $\varepsilon=(-1)^{t/r}$.
\end{proposition}

\begin{proof} For $c\in\Fqdo$, let 
	$$\Ca_c=\{(x,y)\in\Fqdo^2:a_1 x_1^{d_1}+a_2 x_2^{d_2}=c\}$$
	be the set of solutions of $a_1 x_1^{d_1}+a_2 x_2^{d_2}=c$. Assume $c=0$. A pair $(x_1,x_2)=(\alpha^i,\alpha^j)$ is a solution of $a_1 x_1^{d_1}+a_2 x_2^{d_2}=0$ if and only if $\alpha^{id_1-jd_2}=-\tfrac{a_2}{a_1}$ and it is easy to verify that this occurs if and only if $\theta_l(a_1,a_2)=1$. If $\theta_l(a_1,a_2)=0$, then $(x_1,x_2)=(0,0)$ is the unique solution of $a_1 x^{d_1}+a_2 y^{d_2}=0$. Otherwise, there exists an integer $k$ such that $-\tfrac{a_2}{a_1}=\alpha^{lk}$ and then 
	$$\Ca_0=\bigcup_{\lambda=0}^{l-1}\left\{(\alpha^i,\alpha^j):\alpha^{i\tfrac{d_1}{l}+j\tfrac{d_2}{l}}=\alpha^{\lambda\frac{q^2-1}{l}}\alpha^k\right\}\bigcup\{(0,0)\},$$
	where the sets in the union are disjoint. Let $A_\lambda=\{(\alpha^i,\alpha^j):\alpha^{i\tfrac{d_1}{l}+j\tfrac{d_2}{l}}=\alpha^{\lambda\frac{q^2-1}{l}}\alpha^k\}$. We note that
	$$|A_\lambda|=|\{(i,j)\in\Z_{(q^2-1)}:i\tfrac{d_1}{l}+j\tfrac{d_2}{l}\equiv \lambda\tfrac{q^2-1}{l}+k\pmod{q^2-1}\}|.$$
	Let $(i_0,j_0)$ be integers such that $i_0 \tfrac{d_1}{l}+j_0\tfrac{d_2}{l}=1$. For $u\in\Z_{(q^2-1)}$, let
	\begin{equation}\label{item205}
	B_u:=\{(i,j)\in\Z_{(q^2-1)}:i\tfrac{d_1}{l}+j\tfrac{d_2}{l}\equiv u\pmod{q^2-1}\}.
	\end{equation}
	For $u,v\in\Z_{(q^2-1)}$, it is easy to verify that the function $\varphi:B_u\rightarrow B_v$ defined by $\varphi:(i,j)\mapsto(i+i_0(v-u),j+j_0(v-u))$ is a bijective function from $B_u$ to $B_v$ and so $|B_u|=|B_v|$. As $u$ and $v$ are arbitrarily taken, it follows that $\sum_u |B_u|=(q-1)^2$, then $|B_u|=q-1$ for all $u\in\Fqdo$. In particular, $|A_\lambda|=|B_{\lambda(q^2-1)/l+k}|=q-1$ and so $|C_0|=l(q-1)+1$, proving the first part of our result for the case $\theta_l(a_1,a_2)=1$. 
	
	Assume $c\neq0$ and let $B_0$ the set as defined in Eq.~\ref{item205}. We have that
	$$\begin{aligned}|\Ca_c|&=\sum_{b_1+b_2=1}\left[1+\cdots+\chi_{d_1}^{d_1-1}\left(\tfrac{bb_1}{a_1}\right)\right]\left[1+\cdots+\chi_{d_2}^{d_2-1}\left(\tfrac{bb_2}{a_2}\right)\right]\\
	&=p^{2t}+\sum_{b_1+b_2=1}\sum_{1\le \ell_i\le d_i-1} \chi_{d_1}^{\ell_1}\left(\tfrac{bb_1}{a_1}\right)\chi_{d_2}^{\ell_2}\left(\tfrac{bb_2}{a_2}\right)\\
	&=p^{2t}+\sum_{\substack{1\le \ell_i\le d_i-1\\ (\ell_1,\ell_2)\not\in B_0}}\sum_{b_1+b_2=1} \chi_{d_1}^{\ell_1}\left(\tfrac{bb_1}{a_1}\right)\chi_{d_2}^{\ell_2}\left(\tfrac{bb_2}{a_2}\right)+\sum_{\substack{1\le \ell_i\le d_i-1\\ (\ell_1,\ell_2)\in B_0}}\sum_{b_1+b_2=1}\chi_{d_1}^{\ell_1}\left(\tfrac{bb_1}{a_1}\right)\chi_{d_2}^{\ell_2}\left(\tfrac{bb_2}{a_2}\right)\\
	&=p^{2t}+\varepsilon p^t\sum_{\substack{1\le \ell_i\le d_i-1\\ (\ell_1,\ell_2)\not\in B_0}}\chi_{d_1}^{\ell_2}\left(\tfrac{b}{a_1}\right)\chi_{d_2}^{\ell_2}\left(\tfrac{b}{a_2}\right)-\sum_{\substack{1\le \ell_i\le d_i-1\\ (\ell_1,\ell_2)\in B_0}}\chi_{d_1}^{\ell_1}\left(\tfrac{b}{a_1}\right)\chi_{d_2}^{\ell_2}\left(\tfrac{b}{a_2}\right).\\
	\end{aligned}$$
	The last equality follows from Lemma~\ref{item206} and entails that
	\begin{equation}\label{item207}
	\begin{aligned}
	|\Ca_c|&=p^{2t}+\varepsilon p^t\sum_{1\le \ell_i,\le d_i-1}\chi_{d_1}^{\ell_2}\left(\tfrac{b}{a_1}\right)\chi_{d_2}^{\ell_2}\left(\tfrac{b}{a_2}\right)-(\varepsilon p^t-1)\sum_{\substack{1\le \ell_i,\le d_i-1\\ (\ell_1,\ell_2)\in B_0}}\chi_{d_1}^{\ell_1}\left(\tfrac{b}{a_1}\right)\chi_{d_2}^{\ell_2}\left(\tfrac{b}{a_2}\right)\\
	&=q^{2}+\varepsilon q(1-d_1)^{\theta_{d_1}(a_1,b)}(1-d_2)^{\theta_{d_2}(a_2,b)}-(\varepsilon p^t-1)\sum_{(\ell_1,\ell_2)\in\Lambda}\chi_{d_1}^{\ell_1}\left(\tfrac{b}{a_1}\right)\chi_{d_2}^{\ell_2}\left(\tfrac{b}{a_2}\right),\\
	\end{aligned}
	\end{equation}
	where $\Lambda:=B_0\cap\{(\ell_1,\ell_2):1\leq \ell_i\le d_i-1\text{ for }i=1,2\}$. A direct computation shows that
	$$\Lambda=\{(\tfrac{nd_1}{l},\tfrac{-nd_2}{l}):1\le n\le l-1\}$$
	and therefore
	\begin{equation}\label{item208}
	\sum_{(\ell_1,\ell_2)\in\Lambda}\chi_{d_1}^{\ell_1}\left(\tfrac{b}{a_1}\right)\chi_{d_2}^{\ell_2}\left(\tfrac{b}{a_2}\right)=\sum_{n=1}^{d-1}\chi_l^n\left(\tfrac{a_2}{a_1}\right)=\begin{cases}
	l-1,&\text{ if }\theta_l(a_1,a_2)=1;\\
	-1,&\text{ if }\theta_l(a_1,a_2)=0.\\
	\end{cases}
	\end{equation}
	Our result follows from Eq. \eqref{item207} and \eqref{item208}.
\end{proof}
The following definitions will be useful in the proof of our results.
\begin{definition}\label{item254} We define the following:
	\begin{enumerate}[label=(\alph*)]
		\item For $f\in\Fqdo[x_1\dots,x_s]$, let 
		$$N(f(x_1,\dots,x_s)=0)=|\{(x_1,\dots,x_s)\in\Fqdo^s:f(x_1,\dots,x_s)=0\}|,$$ the number of solutions of the equation $f(x_1,\dots,x_s)=0$ over $\Fqdo$.
		\item For $f\in\Fqdo[x,y]$, let $$N_y^*(f(x,y)=0)=|\{(x,y)\in\Fqdo\times\Fqdo^*:f(x,y)=0\}|,$$ the number of solutions of the equation $f(x,y)=0$ with $y\neq0$.
	\end{enumerate}

\end{definition}

The following result is straightforward.
\begin{lemma}\label{item215}
	Let $a,b\in\Fqdo$ and let $\alpha$ be a primitive element of $\Fqdo$. Let $d_1$ and $d_2$ be divisors of $q^2-1$. Then
	$$d_2^{-1}\cdot N_y^*\left(a_1x^{d_1}+\alpha^j y^{d_2}=b\right)=\sum_{i=1}^{(q-1)/d_2}N\left(a_{1}x^{d_{1}}=b-\alpha^{j+id_2}\right)$$
\end{lemma}

The following lemma will be important in the proof of Theorem \ref{item201}. 

\begin{lemma}\label{item217} Let $\varepsilon\in\{\pm 1\}$ and $d,d_1,\dots,d_s$ be integers such that $d|(q-\varepsilon)$ and $d_i|(q-\varepsilon)$ for all $i=1,\dots,s$.  
	\begin{enumerate}[label=(\alph*)]
		\item\label{item255} For $b\in\Fqdo^*$, we have the following relation:
		$$\sum_{j=1}^{q-\varepsilon} (1-d)^{\theta_d(\alpha^j,b)}=0.$$
		\item\label{item256} Let $M_s(\vec{a},\vec{d},q^2,\alpha^j)$ denote the right-hand expression in Eq.~\eqref{item218} of Theorem~\ref{item201}. Then
		$$\sum_{j=1}^{q-\varepsilon} M_s(\vec{a},\vec{d},q^2,\alpha^j)=q^{2(s-1)}(q-\varepsilon)-\varepsilon^{s}q^{s-2}\sum_{j=1}^{q-\varepsilon}\prod_{i=1}^{s}(1-d_i)^{\delta_{i,j}}.$$
	\end{enumerate}
\end{lemma}

\begin{proof}
	We observe that 
	$$-(1-d)^{\theta_d(\alpha^j,b)}=\chi_d\left(\frac{\alpha^j}{b}\right)+\dots+\chi_d^{d-1}\left(\frac{\alpha^j}{b}\right).$$ Therefore, for  $b\in\Fqdo^*$, it follows that
	$$-\sum_{j=1}^{q-\varepsilon} (1-d)^{\theta_d(\alpha^j,b)}=\sum_{j=1}^{q-\varepsilon}\left[\chi_d\left(\frac{\alpha^j}{b}\right)+\dots+\chi_d^{d-1}\left(\frac{\alpha^j}{b}\right)\right]=0,$$
	which proves item~\ref{item255}.
	For an integer $s\ge 2$, it follows that
	$$\begin{aligned}
	\sum_{j=1}^{q-\varepsilon} M_s(\vec{a},\vec{d},q^2,\alpha^j)&=(q-\varepsilon)\!\Bigg(\!q^{2s-2}-\sum_{j=1}^{q-\varepsilon}\prod_{i=1}^{s}(1-d_i)^{\delta_{i,j}}\!\Bigg)-\varepsilon^{s}q^{s-2}\sum_{j=1}^{q-\varepsilon}\prod_{i=1}^{s}(1-d_i)^{\delta_{i,j}}\\
	&= q^{2(s-1)}(q-\varepsilon)-\varepsilon^{s}q^{s-2}\sum_{j=1}^{q-\varepsilon}\prod_{i=1}^{s}(1-d_i)^{\delta_{i,j}},\\
	\end{aligned}$$
	proving item~\ref{item256}.
\end{proof}

From the previous results we are able to prove our main result.

\subsection{Proof of Theorem \ref{item201}}  We proceed by induction on $s$. We make an abuse of language using $\vec{a}$ and $\vec{d}$ for all $k\le s$; meaning that, for each $k$, we use only the first $k$ entries of $\vec{a}$ and $\vec{d}$. The base case of the induction is $s = 2$, which we prove as follows.

\subsubsection{The case $s=2$} By Proposition~\ref{item212}, we only need to show that 
\begin{equation}\label{item213}
(1-l)^{\theta_l(a_1,a_2)}=-(q-\varepsilon)^{-1}\sum_{j=1}^{q-\varepsilon}(1-d_1)^{\theta_{d_1}(a_1,\alpha^j)}(1-d_2)^{\theta_{d_2}(a_2,\alpha^j)},
\end{equation}
where $l=\gcd(d_1,d_2)$. By the Inclusion and Exclusion Principle, the sum in Eq.~\eqref{item213} is equal to
$$\begin{aligned}
&(1-d_1)(1-d_2)n+\sum_{i=1}^{2}(1-d_i)(\tfrac{q-\varepsilon}{d_i}-n)+\left(q-\varepsilon-(\tfrac{q-\varepsilon}{d_1}-n)-(\tfrac{q-\varepsilon}{d_2}-n)-n\right)\\
&=d_1d_2n-(q-\varepsilon),\\
\end{aligned}$$
where $n:=|\{1\le j\le q-\varepsilon:\theta_{d_1}(a_1,\alpha^j)=\theta_{d_2}(a_2,\alpha^j)=1\}|$.
We split the proof for Eq.~\eqref{item213} into two cases:
\begin{itemize}
	\item The case $\theta_l(a_1,a_2)=0$. Since $a_1^{(q^2-1)/l}\neq a_2^{(q^2-1)/l}$, it follows that there exists no $j$ such that $a_1^{(q^2-1)/d_1}=(\alpha^j)^{(q^2-1)/d_1}$ and $a_2^{(q^2-1)/d_2}=(\alpha^j)^{(q^2-1)/d_2}$ and then $n=0$. In this case, $d_1d_2n-(q-\varepsilon)=-(q-\varepsilon)$;
	\item The case $\theta_l(a_1,a_2)=1$. Let $j_1$ and $j_2$ be integers such that $a_1=\alpha^{j_1}$ and $a_2=\alpha^{j_2}$. Since $a_1^{(q^2-1)/l}= a_2^{(q^2-1)/l}$, it follows that $j_1\equiv j_2\pmod{l}$ and this relation entails that
	$$\{1\le j\le q-\varepsilon:\theta_{d_i}(a_i,\alpha^j)=1\text{ for }i=1,2\}=\left\{k\lcm(d_1,d_2):1\le k\le \tfrac{q-\varepsilon}{\lcm(d_1,d_2)} \right\}$$
	and then $n=\tfrac{q-\varepsilon}{\lcm(d_1,d_2)}$. In this case, $d_1d_2n-(q-\varepsilon)=(q-\varepsilon)(l-1)$;
\end{itemize}
Altogether, we have shown that $d_1d_2n-(q-\varepsilon)=-(q-\varepsilon)(1-l)^{\theta_l(a_1,a_2)}$ and then
$$(q-\varepsilon)^{-1}\sum_{j=1}^{q-\varepsilon}(1-d_1)^{\theta_{d_1}(a_1,\alpha^j)}(1-d_2)^{\theta_{d_2}(a_2,\alpha^j)}=-(1-l)^{\theta_l(a_1,a_2)},$$
therefore our result follows for the case $s=2$.

\subsubsection{Induction Hypothesis} Suppose that the result holds for $N_s(\vec{a},\vec{d},q^2,b)$ with $s\le k$ and $b\in\Fqdo^*$. We observe that 
$$\begin{aligned}N_{k+1}(\vec{a},\vec{d},q^2,b)
&=\sum_{c\in\Fqdo} N(a_1 x_1^{d_1}+\cdots+a_k x_k^{d_{k}}=c)N(b-a_{k+1}x_{k+1}^{d_{k+1}}=c)\\
&=\sum_{c\in\Fqdo} N_k(\vec{a},\vec{d},q^2,c)N(b-a_{k+1}x_{k+1}^{d_{k+1}}=c).\\
\end{aligned}$$
Along the proof, we denote $N_k(\vec{a},\vec{d},q^2,c)$ by $M_c$. Let $\theta_b=\theta_{d_{k+1}}(a_{k+1},b)$ and $C_0=(1-(1-d_{k+1})^{\theta_b})M_0$. Since $M_{\alpha^i}=M_{\alpha^j}$ if $i\equiv j\pmod{q-\varepsilon}$, it follows that
{\small\begin{equation}\label{item216}
\begin{aligned}N_{k+1}(\vec{a},\vec{d},q^2,b)&=N(b=a_{k+1}x_{k+1}^{d_{k+1}})M_0 +\sum_{j=1}^{q-\varepsilon} M_{\alpha^j}\sum_{i\equiv j}N(b-a_{k+1}x_{k+1}^{d_{k+1}}=\alpha^i)\\
&=(1-(1-d_{k+1})^{\theta_b})M_0+\sum_{j=1}^{q-\varepsilon} M_{\alpha^j}\sum_{i\equiv j}N(b-a_{k+1}x_{k+1}^{d_{k+1}}=\alpha^i)\\
&=C_0+\sum_{j=1}^{q-\varepsilon} \left(\frac{M_{\alpha^j}}{q-\varepsilon}\right)N_y^*(b-a_{k+1}x_{k+1}^{d_{k+1}}=\alpha^j y^{q-\varepsilon}),
\end{aligned}
\end{equation}
where the last equality follows from Lemma~\ref{item215} and $N_y^*(b-a_{k+1}x_{k+1}^{d_{k+1}}=\alpha^j y^{q-\varepsilon})$ is defined as in Definition \ref{item254}. Let $M_j^*=N_y^*(b-a_{k+1}x_{k+1}^{d_{k+1}}=\alpha^j y^{q-\varepsilon})$ for $1\le j\le q-\varepsilon$. Let $B$ be a positive integer such that $b=\alpha^{B}$, where $B=(q-\varepsilon)\ell+m$ for some non-negative integer $\ell$ and some $m< q-\varepsilon$. Proposition~\ref{item212} entails that
{\small$$M_j^*=\hspace{-0.3em} \begin{cases}&\hspace{-1em} q^2-1+(1-\varepsilon q)(1-d_{k+1})^{\theta_b}-\varepsilon(q-\varepsilon)(1-d_{k+1})^{\theta_{d_{k+1}}(\alpha^j,b)},\text{ if }j\neq m;\\
&\hspace{-1em}q^2-1+(\varepsilon q-1)(q-1)(1-d_{k+1})^{\theta_b}-\varepsilon(q-\varepsilon)(1-d_{k+1})^{\theta_{d_{k+1}}(\alpha^j,b)},\text{ if }j=m.\\
\end{cases}$$}
Let $\theta_j=\theta_{d_{k+1}}(\alpha^j,b)$ for $1\leq j\le q-\varepsilon$. Then from Eq.~\eqref{item216} it follows that
{\small$$N_{k+1}(\vec{a},\vec{d},q^2,b)=C_0+\sum_{j=1}^{q-\varepsilon} \tfrac{M_{\alpha^j}(q^2-1+(1-\varepsilon q)(1-d_{k+1})^{\theta_b}-\varepsilon(q-\varepsilon)(1-d_{k+1})^{\theta_j})}{q-\varepsilon}+\varepsilon q(1-d_{k+1})^{\theta_b}M_b.$$}
Set $S_j=\prod_{i=1}^{k}(1-d_i)^{\nu_i(\alpha^j)}$ and $S=\sum_{j=1}^{q-\varepsilon}\prod_{i=1}^{k}(1-d_i)^{\delta_{i,j}}$. By the induction hypothesis and Lemma~\ref{item217}, the last equality becomes

\begin{equation}\label{item224}N_{k+1}(\vec{a},\vec{d},q^2,b)= C_0+C_1+\varepsilon^{k}q^{k-1}\sum_{j=1}^{q-\varepsilon}S_j(1-d_{k+1})^{\theta_j}+C_2.
\end{equation}
where $C_1:=\tfrac{\left(q^{2(k-1)}(q-\varepsilon)-\varepsilon^{k}q^{k-2}S\right)\left(q^2-1+(1-\varepsilon q)(1-d_{k+1})^{\theta_b}\right)}{q-\varepsilon}$ and $C_2:=\varepsilon q(1-d_{k+1})^{\theta_b}M_b$. We recall that the values of $M_b$ and $M_0$ are known (by the induction hypothesis and Corollary~\ref{item250}) and so a straightforward computation shows that $C_0+C_2$ equals
{\small\begin{equation}\label{item225}
q^{2(k-1)}+\varepsilon^k q^{k-2}(q+\varepsilon)S-(1-d_{k+1})^{\theta_b}\Bigg(\!(1-\varepsilon q)q^{2(k-1)}+\varepsilon^{k+1}q^{k-2}S+\varepsilon^k q^{k}\prod_{i=1}^{k}(1-d_i)^{\nu_i(b)}\Bigg)
\end{equation}}

By Eq.~\eqref{item224} and \eqref{item225}, it follows that
$$\begin{aligned}N_{k+1}(\vec{a},\vec{d},q^2,b)&=q^{2k}-\varepsilon^k q^{k-1}\Bigg(q(1-d_{k+1})^{\theta_b}\prod_{i=1}^{k}(1-d_i)^{\nu_i(b)}-\sum_{j=1}^{q-\varepsilon}S_j(1-d_{k+1})^{\theta_j}\Bigg)\\
&=q^{2k}-\varepsilon^k q^{k-1}\Bigg(q\prod_{i=1}^{k+1}(1-d_i)^{\nu_i(b)}-\sum_{j=1}^{q-\varepsilon}\prod_{i=1}^{k+1}(1-d_i)^{\delta_{i,j}}\Bigg),\\
\end{aligned}$$
which proves the result for $s=k+1$.$\hfill\qed$\\

Let $\mathcal{F}(n,m)$ be the projective Fermat type curve given by the affine equation $ax^n+by^m=c$. As a direct consequence Theorem~\ref{item201}, we have the number of points on $\mathcal{F}(n,m)$ in the case where there exists a divisor $r$ of $t$ such that $n$ and $m$ divide $p^r+1$. In fact, Corollary~\ref{item202} follows easily from Theorem~\ref{item201} by taking account the points at the infinity. From Example 6.3.3 in~\cite{stichtenoth2009algebraic}, it follows that the genus $g$ of $\mathcal{F}(n,m)$ is given by
$$g=\frac{(n-1)(m-1)+1-\gcd(n,m)}{2}.$$
 Corollary~\ref{item202} provides conditions in which the number of $\Fq$-rational points on $\mathcal{F}(n,m)$ attains the Hasse-Weil bound. In fact, the characterization of maximal and minimal varieties is a problem that has been extensively studied in the last few decades. In Section~\ref{item221} we will study this problem in the case where the variety is given by a diagonal equation. Before doing that, we will study the number of solutions attaining Weil's bound in the affine space.

\section{On Maximality and Minimality}\label{item242}

As a consequence of Theorem~\ref{item201}, we get an explicit formula for the number $I(d_1,\dots,d_s)$ defined in the introduction. 
\begin{remark}\label{item226}
	Let $r$ be a positive integer and let $p$ be a prime. Let $d_1,\dots,d_s$ be positive integers such that $d_i|(p^r+1)$. By Theorem 2.9 in \cite{cao2016number} and Theorem~\ref{item201}, it follows that
	$$I(d_1,\dots,d_s)=\tfrac{(-1)^s}{p^r+1}\sum_{j=1}^{p^r+1}\prod_{i=1}^{s}(1-d_i)^{\delta_{i,j}},$$
	where $\delta_{i,j}=1$ if $d_i|j$ and $\delta_{i,j}=0$ otherwise. Therefore
	$$I(d_1,\dots,d_s)=\tfrac{(-1)^s}{p^r+1}\sum_{m=1}^{p^r+1}\prod_{d_i|m} (1-d_i).$$
\end{remark}

Indeed, the formula in Remark~\ref{item226} is also true in a more general setting. 
\begin{lemma}\label{item232}
	Let $d_1,\dots,d_s$ be positive integers. Then
	\begin{equation}\label{item252}
	I(d_1,\dots,d_s)=\tfrac{(-1)^s}{D}\sum_{m=1}^{D}\prod_{d_i|m} (1-d_i),
	\end{equation}
	where $D:=\lcm(d_1,\dots,d_s)$. 
\end{lemma}

\begin{proof}
	A well-known formula for $I(d_1,\dots,d_s)$ (see p.293 of \cite{Lidl}) is the following:
	\begin{equation}\label{item253}
	I(d_1,\dots,d_s)=(-1)^s+(-1)^{s}\sum_{r=1}^s (-1)^{r}\sum_{1\le i_1<\dots<i_r\le s}\frac{d_{i_1}\dots d_{i_r}}{\lcm(d_{i_1},\dots,d_{i_r})}.
	\end{equation}
	Let $i_1,\dots,i_r$ be integers such that $1\le i_1<\dots<i_r\le s$. The product $d_{i_1}\cdots d_{i_r}$ appears in an expansion of a product in \eqref{item252} whenever $d_{i_j}|m$ for all $j=1,\dots,r$, which occurs $\tfrac{D}{\lcm(d_{i_1},\dots,d_{i_r})}$ times, since $1\le m\le D$. Therefore the expressions in \eqref{item252} and \eqref{item253} coincide, proving our result.
\end{proof}

From here, Lemma~\ref{item251} is a simple employment of Lemma~\ref{item232}. The formula in Lemma~\ref{item232} had already been established in Lemma 9 of \cite{small1984diagonal} with a different proof.   Most notably, sufficient and necessary conditions in which $I(d_1,\dots,d_s)=0$ were studied by Sun and Wan \cite{sun1987solvability}, where the authors state the following result.

\begin{lemma}\cite[Theorem]{sun1987solvability}\label{item240}
	Let $s>2$ be an integer and let $d_1,\dots,d_s$ be positive integers. Then $I(d_1,\dots,d_s)=0$ if and only if one of the following holds:
	\begin{itemize}
		\item for some $d_i$, $\gcd(d_i,d_1\dots d_s/d_i)=1$, or
		\item if $d_{i_1},\dots,d_{i_k}$ $(1\le i_1<\dots i_k\le s)$ is the set of all even integers among $\{d_1,\dots,d_s\}$, then $2\nmid k$, $d_{i_1}/2,\dots,d_{i_k}/2$ are pairwise prime, and $d_{i_j}$is prime to any odd number in $\{d_1,\dots,d_s\}\, (j=1,\dots,k)$.
	\end{itemize}
	
\end{lemma}

 In this section we will work on diagonal equations over $\Fq$, where $q=p^t$. For $l$ a divisor of $q-1$, let $\chi_l$ be a multiplicative character of $\Fq^*$ of order $l$. In order to bound the number of solutions of diagonal equations, let us recall a well-known way to compute such number. 
\begin{equation}\label{item235}
\begin{aligned}N_s(\vec{a},\vec{d},q,b)&=\sum_{b_1+\dots+b_s=b}\, \prod_{i=1}^s\left[1+\chi_{d_i}(a_i^{-1}b_i)+\dots+\chi^{d_i-1}_{d_i}(a_i^{-1}b_i)\right]\\
&=q^{s-1}+\sum_{0< \ell_i<d_i} \chi_{d_1}^{\ell_1}(a_1^{-1})\dots \chi_{d_s}^{\ell_s}(a_s^{-1})J_q(\chi_{d_1}^{\ell_1},\dots,\chi_{d_s}^{\ell_s},b).\\
\end{aligned}
\end{equation}
From here, since we want to estimate $N_s(\vec{a},\vec{d},q,b)$, an interesting problem is to find necessary and sufficient conditions for $\chi_1^{\ell_1}(a_1^{-1})\dots \chi_s^{\ell_s}(a_s^{-1})J_q(\chi_1^{\ell_1},\dots,\chi_s^{\ell_s},b)$ being real. In order to do that, we use the concept of purity of Jacobi sums. A Jacobi sums is said to be {\it pure} if some non-zero integral power of it is real. The purity of Jacobi sums have been extensively studied~\cite{akiyama1996pure,oliver2011gauss,evans1981pure}. The following result concerning this concept will be useful in the proof of our results. From now, we restrict ourselves to the case $d_1=\dots=d_s$.

\begin{lemma}\cite[Proposition 3.5 and Remark 3.8]{shioda1979fermat}\label{item234}
	Let $s\ge 2$ be an integer, $d\geq 3$ be a divisor of $q-1$ and let
	$$\mathcal{U}_s(d)=\{(\ell_1,\dots,\ell_s)\in[1,d-1]^s:\ell_1+\dots+\ell_s\not\equiv 0\pmod{d}\}.$$
	If $s=3$, assume that $d>3$. If the Jacobi sum $J_q(\chi_{d}^{\ell_1},\dots,\chi_{d}^{\ell_s},1)$ is pure for all $s$-tuples $(\ell_1,\dots,\ell_s)\in\mathcal{U}_s(d)$, then there exists an integer $r$ such that $d|(p^r+1)$. 
\end{lemma}

\subsection{Proof of Theorem \ref{item219}}

Let $\mathcal{U}_s(d)$ be as defined in Lemma~\ref{item234}. Suppose that $b=0$, $s\ge 3$ and assume that the bound \ref{item236} is attained. By Eq.~\eqref{item235} and Proposition~\ref{item203}, we have that
\begin{equation}\label{item238}
N_s(\vec{a},\vec{d},q,0)=q^{s-1}+\sum_{\substack{0< \ell_i<d\\ (\ell_1,\dots,\ell_s)\in\mathcal{U}_s^c(d)}} \chi_{d}^{\ell_1}(a_1^{-1})\dots \chi_{d}^{\ell_s}(a_s^{-1})J_q(\chi_{d}^{\ell_1},\dots,\chi_{d}^{\ell_s},0),
\end{equation}
where $\mathcal{U}_s^c(d):=\{(\ell_1,\dots,\ell_s)\in[1,d-1]^s:\ell_1+\dots+\ell_s\equiv 0\pmod{d}\}$. Moreover, if $ (\ell_1,\dots,\ell_s)\in\mathcal{U}_s^c(d)$, then
\begin{equation}\label{item239}
\begin{aligned}J_q(\chi_{d}^{\ell_1},\dots,\chi_{d}^{\ell_s},0)&=\sum_{b_1+\dots+b_s=0}\chi_{d}^{\ell_1}(b_1)\dots\chi_{d}^{\ell_s}(b_s)\\
&=(q-1)\sum_{b_1+\dots+b_{s-1}=1}\chi_{d}^{\ell_1}(b_1)\dots\chi_{d}^{\ell_{s-1}}(b_{s-1})\\
&=(q-1)J_q(\chi_{d}^{\ell_1},\dots,\chi_{d}^{\ell_{s-1}},1).\\
\end{aligned}
\end{equation}
We recall that $|\mathcal{U}_s^c(d)|=I(d,\dots,d)$ and that we are assuming that the bound \ref{item236} is attained. Furthermore, by Lemma~\ref{item240}, it follows that $I(d,\dots,d)=0$ if and only if $d=2$ and $s$ is odd. Therefore, since we are under the assumption $d>2$, it follows from Eq.~\eqref{item238} and \eqref{item239} that 
$$\chi_{d}^{\ell_1}(a_1^{-1})\dots \chi_{d}^{\ell_s}(a_s^{-1})J_q(\chi_{d}^{\ell_1},\dots,\chi_{d}^{\ell_{s-1}},1)\in\{\pm q^{(s-2)/2}\}$$
for all $(\ell_1,\dots,\ell_s)\in\mathcal{U}_s^c(d)$ or, equivalently, for all $(\ell_1,\dots,\ell_{s-1})\in\mathcal{U}_{s-1}(d)$. In particular, $J_q(\chi_{d}^{\ell_1},\dots,\chi_{d}^{\ell_{s-1}},1)$ is pure for all $(\ell_1,\dots,\ell_{s-1})\in\mathcal{U}_{s-1}(d)$ and so, by Lemma~\ref{item234}, it follows that there exists an integer $r$ such that $d|(p^r+1)$. Assume that $r$ is the smallest integer such that $d|(p^r+1)$. Then $2r$ is the order of $p$ in the multiplicative group $\Z_d^\times$. Since $q\equiv 1\pmod{d}$, with $q=p^n$, it follows that $2r|n$. In particular, $n=2t$ for some $t$ multiple of $r$. Therefore we are under the hypothesis of Theorem~\ref{item201} and so
	$$N_s(\vec{a},\vec{d},q,0)=q^{2(s-1)}+\varepsilon^s q^{s-2}(q+\varepsilon)\sum_{j=1}^{q-\varepsilon}\prod_{i=1}^{s}(1-d)^{\delta_{i,j}},$$
where $\varepsilon=(-1)^{t/r}$. It is direct to verify that $|(-1)^s(q-\varepsilon)^{-1}\sum_{j=1}^{q-\varepsilon}\prod_{i=1}^{s}(1-d)^{\delta_{i,j}}|=|I(d,\dots,d)|$ if and only if $\chi_d(a_1)=\dots=\chi_d(a_s)$, proving our result. Moreover, 
$$(-1)^s(q-\varepsilon)^{-1}\sum_{j=1}^{q-\varepsilon}\prod_{i=1}^{s}(1-d)^{\delta_{i,j}}=-I(d,\dots,d)$$
 if and only if $\varepsilon=1$ and $s$ is odd, which occurs if and only if $t/r$ is even and $s$ is odd. The converse follows from Theorem~\ref{item201}. 

 The case $b\neq 0$ can be obtained similarly to the case $b=0$.$\hfill\qed$\newline
 
 Theorem \ref{item219} does not consider the case where $s=2$ and $b=0$. For $s=2$ only the upper bound can be attained, as it is shown in the following remark.
 
\begin{remark}\label{item248}
	 If $s=2$ and $b=0$, then $I(d,d)=d-1$ and
	$$J_q(\chi_{d}^{\ell_1},\chi_{d}^{\ell_2},0)=\sum_{b_1\in\Fq}\chi_d(b_1^{\ell_1}(-b_1)^{\ell_2})=
	\begin{cases}
	0,&\text{if }d\nmid(\ell_1+\ell_2);\\
	(q-1)\chi_d((-1)^{\ell_2}),&\text{if }d|(\ell_1+\ell_2).\\
	\end{cases}$$
	Furthermore, $d|(\ell_1+\ell_2)$ if and only if $\ell_1=d-\ell_2$ with $\ell_2=1,\dots,d-1$. Therefore, by Eq.~\eqref{item238} we have that
	$$\begin{aligned}
	N_2(\vec{a},\vec{d},q,0)&=q+\sum_{\ell_2=1}^{d-1}\chi_{d}^{-\ell_2}(a_1^{-1})\chi_{d}^{\ell_2}(a_2^{-1})(q-1)\chi_d((-1)^{\ell_2})\\
	&=q+(q-1)\sum_{\ell_2=1}^{d-1}\chi_{d}^{-\ell_2}(a_1^{-1})\chi_{d}^{\ell_2}(-a_2^{-1})\\
	\end{aligned}$$
	and then the bound \ref{item236} is attained if and only if $\chi_{d}(-a_1a_2^{-1})=1$. Therefore the bound in Theorem~\ref{item219} is attained only if $\chi_{d}(a_1)=\chi_{d}(-a_2)$ and, in this case, the upper bound is attained. 
\end{remark}

\section{Diagonal Equations in the Projective Space}\label{item221}
Let $f\in\Fq[x_1,\dots,x_s]$ be a homogeneous polynomial of degree $d$ and let $V$ be the projective variety defined by $f$. Let $V(\Fq)$ be the number of points of $V$ over $\Fq$. If $V$ is absolutely irreducible and non-singular, then the Weil-Deligne bound states that 
$$\left|V(\Fq)-\frac{q^{s-1}-1}{q-1}\right|\le q^{(s-2)/2}B(d,s),$$
where $B(d,s)=\tfrac{(d-1)^s+(-1)^s(d-1)}{d}$. In the case where $s=3$, this bound is the well-known Hasse-Weil bound. Projective varieties attaining the Weil-Deligne bound are called maximal and minimal. Let $f(x_1,\dots,x_s)=a_1 x_1^d+\dots+a_s x_s^d$ and let $V_{s,d}$ be the variety defined by $f(x_1,\dots,x_s)=0$. It is easy to verify that 
$$(q-1)V_{s,d}(\Fq)+1=N_s(\vec{a},\vec{d},q,0)$$
where $\vec{a}=(a_1,\dots,a_s)$ and $\vec{d}=(d,\dots,d)$. Corollary~\ref{item250} provides the number of projective points on $V_{s,d}$ for a family of polynomials $f$ satisfying some conditions under their degree. In the following result we employ Corollary~\ref{item250} and Theorem~\ref{item219}  in order to present conditions under $f$ in which the Weil-Deligne bound is attained.

\begin{corollary}\label{item241}
	Assume that $s\geq3$ and $d>2$ are integers such that $(s,d)\neq(4,3)$ and let $V_{s,d}$ be the projective variety defined by the equation $a_1 x_1^d+\dots +a_s x_s^d=0$. Then the following hold.
	\begin{enumerate}
		\item $V_{s,d}$ is maximal over $\Fq$ if and only if 
		\begin{itemize}
			\item $q=p^{2t}$ for some positive integer $t$;
			\item there exists a divisor $r$ of $t$ such that $d|(p^{r}+1)$;
			\item either $t/r$ is odd or $s$ is even;
			\item $\chi_{d}(a_1)=\dots=\chi_{d}(a_s)$.
		\end{itemize}
		\item  $V_{s,d}$ is minimal over $\F_{q^2}$ if and only if 
		\begin{itemize}
			\item $q=p^{2t}$ for some positive integer $t$;
			\item $t$ is even and there exists a divisor $r$ of $t/2$ such that $d|(p^{r}+1)$;
			\item $s$ is odd;
			\item $\chi_{d}(a_1)=\dots=\chi_{d}(a_s)$.
		\end{itemize}
	\end{enumerate}
\end{corollary}  

Corollary~\ref{item222} is the particular case $s=3$ of Corollary~\ref{item241}.


\section{Final Comments and Open Problems}\label{item246}

This paper provided a counting on the number of solutions of diagonal equations of the form $a_1 x_1^{d_1}+\dots+a_s x_s^{d_s}=b$ satisfying conditions on the exponents. Our counting results extend the main result of \cite{wolfmann1992number} and some results of \cite{cao2016number}. In the general situation where no conditions are imposed under the exponents, an explicit formula for such number is unknown. Indeed, the problem of counting solutions of diagonal equations in a general setting is still an open problem. Many authors have studied particular cases in the last few years. For example, complicated explicit formulas for the case where $d_1=\dots=d_s\in\{3,4\}$ and $q=p$ are presented in Chapter 10 of \cite{berndt1998gauss}. Still in the small degree case, in \cite{oliveira2019rational} the authors studied the number of solutions of diagonal equations in the case where $s=2$ and $d_i\in\{2,3,4,6\}$ are suitably chosen. From here, we pose the following open problem.

\begin{problem}
	Finding an explicit formula for $N_s(\vec{a},\vec{d},q^2,b)$ (with $b\neq0$) for exponents satisfying the conditions of Theorem~\ref{item227}.
\end{problem}

In Section~\ref{item242} we studied the purity of the Jacobi sums arising from the counting of solutions of diagonal equations. In fact, this sums are the eigenvalues of the
endomorphism of the $\ell$-adic \'{e}tale cohomology group induced by the Frobenius morphism of Fermat type varieties given by diagonal equations. When all eigenvalues are pure Jacobi sums, the variety is called {\it supersingular}. Supersingular Fermat varieties was studied by Shioda and Katsura~\cite{shioda1979fermat} and Yui~\cite{yui1980jacobian}. From here, an interesting question is the following. 

\begin{problem}
	What are the conditions on the exponents in which all Jacobi sums arising from the counting of solutions of~\eqref{item200} are pure?
\end{problem}
This is an unsolved problem even in the case $s=2$.

\section{Acknowledgments}

I am grateful to Fabio Enrique Brochero Mart\'{i}nez for his helpful suggestions and I thank him for his useful comments. This study was financed in part by the Coordena\c{c}\~{a}o de Aperfei\c{c}oamento de Pessoal de N\'{i}vel Superior - Brasil (CAPES) - Finance Code 001. 

\bibliographystyle{abbrv}
\bibliography{biblio}

\begin{thebibliography}{10}

\bibitem{akiyama1996pure}
S.~Akiyama.
\newblock On the pure {Jacobi} sums.
\newblock {\em Acta Arithmetica}, 75(2):97--104, 1996.

\bibitem{baoulina2010number}
I.~Baoulina.
\newblock On the number of solutions to certain diagonal equations over finite
  fields.
\newblock {\em International Journal of Number Theory}, 6(01):1--14, 2010.

\bibitem{baoulina2016class}
I.~N. Baoulina.
\newblock On a class of diagonal equations over finite fields.
\newblock {\em Finite Fields and Their Applications}, 40:201--223, 2016.

\bibitem{berndt1998gauss}
B.~C. Berndt, R.~J. Evans, and K.~S. Williams.
\newblock {\em Gauss and Jacobi sums}.
\newblock Wiley New York, 1998.

\bibitem{cao2007factorization}
W.~Cao and Q.~Sun.
\newblock Factorization formulae on counting zeros of diagonal equations over
  finite fields.
\newblock {\em Proceedings of the American Mathematical Society},
  135(5):1283--1291, 2007.

\bibitem{cao2016number}
X.~Cao, W.-S. Chou, and J.~Gu.
\newblock On the number of solutions of certain diagonal equations over finite
  fields.
\newblock {\em Finite Fields and Their Applications}, 42:225--252, 2016.

\bibitem{evans1981pure}
R.~J. Evans.
\newblock Pure {Gauss} sums over finite fields.
\newblock {\em Mathematika}, 28(2):239--248, 1981.

\bibitem{garcia2008certain}
A.~Garcia and S.~Tafazolian.
\newblock Certain maximal curves and {C}artier operators.
\newblock {\em Acta Arithmetica}, 135:199--218, 2008.

\bibitem{hou2009certain}
X.-D. Hou and C.~Sze.
\newblock On certain diagonal equations over finite fields.
\newblock {\em Finite Fields and Their Applications}, 15(6):633--643, 2009.

\bibitem{hua1949characters}
L.~Hua and H.~Vandiver.
\newblock Characters over certain types of rings with applications to the
  theory of equations in a finite field.
\newblock {\em Proceedings of the National Academy of Sciences of the United
  States of America}, 35(2):94, 1949.

\bibitem{Lidl}
R.~Lidl and H.~Niederreiter.
\newblock {\em Finite Fields}, volume~20.
\newblock Cambridge {U}niversity {P}ress, 1997.

\bibitem{Panario}
G.~L. Mullen and D.~Panario.
\newblock {\em Handbook of Finite Fields}.
\newblock Chapman and Hall/CRC, 2013.

\bibitem{oliveira2019rational}
J.~A. Oliveira.
\newblock Rational points on cubic, quartic and sextic curves over finite
  fields.
\newblock {\em arXiv preprint arXiv:1912.11441}, 2019.

\bibitem{oliver2011gauss}
R.~Oliver.
\newblock Gauss sums over finite fields and roots of unity.
\newblock {\em Proceedings of the American Mathematical Society},
  139(4):1273--1276, 2011.

\bibitem{qi1997diagonal}
S.~Qi.
\newblock On diagonal equations over finite fields.
\newblock {\em Finite Fields and Their Applications}, 3(2):175--179, 1997.

\bibitem{shioda1979fermat}
T.~Shioda and T.~Katsura.
\newblock On {F}ermat varieties.
\newblock {\em Tohoku Mathematical Journal, Second Series}, 31(1):97--115,
  1979.

\bibitem{small1984diagonal}
C.~Small.
\newblock Diagonal equations over large finite fields.
\newblock {\em Canadian Journal of Mathematics}, 36(2):249--262, 1984.

\bibitem{stichtenoth2009algebraic}
H.~Stichtenoth.
\newblock {\em Algebraic function fields and codes}, volume 254.
\newblock Springer Science \& Business Media, 2009.

\bibitem{sun1987solvability}
Q.~Sun and D.~Q. Wan.
\newblock On the solvability of the equation $\sum x_i/d_i\equiv 0\pmod{1}$ and
  its application.
\newblock {\em Proceedings of the American Mathematical Society},
  100(2):220--224, 1987.

\bibitem{sun1996number}
Q.~Sun and P.-Z. Yuan.
\newblock On the number of solutions of diagonal equations over a finite field.
\newblock {\em Finite Fields and Their Applications}, 2(1):35--41, 1996.

\bibitem{tafazolian2010characterization}
S.~Tafazolian.
\newblock A characterization of maximal and minimal fermat curves.
\newblock {\em Finite Fields and Their Applications}, 16(1):1--3, 2010.

\bibitem{tafazolian2013maximal}
S.~Tafazolian and F.~Torres.
\newblock On maximal curves of {F}ermat type.
\newblock {\em Advances in Geometry}, 13(4):613--617, 2013.

\bibitem{tafazolian2014curve}
S.~Tafazolian and F.~Torres.
\newblock On the curve $y^n= x^m+ x$ over finite fields.
\newblock {\em Journal of Number Theory}, 145:51--66, 2014.

\bibitem{weil1949numbers}
A.~Weil et~al.
\newblock Numbers of solutions of equations in finite fields.
\newblock {\em Bull. Amer. Math. Soc}, 55(5):497--508, 1949.

\bibitem{wolfmann1992number}
J.~Wolfmann.
\newblock The number of solutions of certain diagonal equations over finite
  fields.
\newblock {\em Journal of Number Theory}, 42(3):247--257, 1992.

\bibitem{yui1980jacobian}
N.~Yui.
\newblock On the {Jacobian} variety of the {Fermat} curve.
\newblock {\em Journal of Algebra}, 65(1):1--35, 1980.

\bibitem{zhou2019counting}
H.~Zhou and Y.~Sun.
\newblock Counting points on diagonal equations over {Galois} rings
  $\text{GR}(p^2, p^{2r})$.
\newblock {\em Finite Fields and Their Applications}, 56:266--284, 2019.

\end{thebibliography}

\end{document}